\documentclass[11pt]{amsart}

\usepackage[latin1]{inputenc}
\usepackage{amssymb,amsmath,amsfonts}
\usepackage{enumerate}

\usepackage[a4paper,twoside,left=3cm,right=2.8cm,top=3.1cm,bottom=2.3cm]{geometry}

\newtheorem{proposition}{Proposition}[section]

\newtheorem{lemma}[proposition]{Lemma}
\newtheorem{theorem}[proposition]{Theorem}
\newtheorem{corollary}[proposition]{Corollary}

\theoremstyle{definition}
\newtheorem{remark}[proposition]{Remark}
\newtheorem{example}[proposition]{Example}

\newcounter{theor}

\newtheorem{teor}[theor]{Theorem}

\DeclareMathOperator{\vol}{Vol}

\DeclareMathOperator{\Gr}{Gr}

\DeclareMathOperator{\SL}{SL}
\DeclareMathOperator{\GL}{GL}

\renewcommand{\O}{\mathrm{O}}
\DeclareMathOperator{\D}{D}

\DeclareMathOperator{\aff}{aff}

\newcommand{\R}{\mathbb{R}}
\newcommand{\C}{\mathbb{C}}

\newcommand{\K}{\mathcal{K}}

\newcommand{\di}{\Diamond}

\newcommand{\st}{\mathrm{st}}

\DeclareMathOperator{\W}{\mathrm{W}}

\def\lin{\mathop\mathrm{lin}\nolimits}
\def\aff{\mathop\mathrm{aff}\nolimits}
\def\conv{\mathop\mathrm{conv}\nolimits}

\newcommand{\func}[5]{\ensuremath{\begin{array}{cccl}
#1:&#2&\longrightarrow&#3\\&#4&\mapsto&#5\end{array}}}

\title[Rogers-Shephard inequality in the characterization of the difference body]{The role of the Rogers-Shephard inequality in the characterization of the difference body}
\author{Judit Abardia} 
\address{Institut f\"ur Mathematik, Goethe-Universit\"at Frankfurt am Main, 
Robert-Mayer-Str. 10, 60054 Frankfurt, Germany}
\email{abardia@math.uni-frankfurt.de}

\author{Eugenia Saor\'in G\'omez} 
\address{Institut für Algebra und Geometrie, Universit\"at Magdeburg, 
Universitätsplatz 2, 39106 Magdeburg, Germany}
\email{eugenia.saorin@ovgu.de}

\begin{document}

\thanks{First author is supported by DFG grant AB 584/1-1.
Second author is supported by Direcci\'on General  de Investigaci\'on MTM2011-25377 MCIT and FEDER}

\date{\today}

\subjclass[2000]{Primary 52A20 
52B45; 
Secondary
52A40
}

\keywords{Difference body, Rogers-Shephard inequality, $\GL(n)$-covariance, $o$-symmetrization, Brunn-Minkowski inequality, Minkowski valuation}

\begin{abstract}
The difference body operator enjoys different characterization results relying on its basic properties such as continuity, $\SL(n)$-covariance, Minkowski valuation or symmetric image. The Rogers-Shephard and the Brunn-Minkowski inequalities provide upper and lower bounds for the volume of the difference body in terms of the volume of the body itself. In this paper we aim to understand the role of the Rogers-Shephard inequality in characterization results of the difference body and, at the same time, to study the interplay among the different properties.
Among others, we prove that the difference body operator is the only continuous and $\GL(n)$-covariant operator from the space of convex bodies to the origin-symmetric ones which satisfies a Rogers-Shephard type inequality while every continuous and $\GL(n)$-covariant operator satisfies a Brunn-Minkowski type inequality. 
\end{abstract}

\maketitle

\section{Introduction}

Let $\K^n$ denote the set of convex bodies (compact and
convex sets) in $\R^n$. 
The support function of a convex body $K\in\K^n$ can be seen as its {\it analytic definition}, as $K$ 
is uniquely determined by means of it
\begin{equation}\label{e: h(K,u)}
h(K, v) = h_K(v)=\max \{ \langle v, x\rangle : x \in K \}.
\end{equation}
Here $v\in \R^n$ and $\langle v, x\rangle$ stands for the standard
inner product of $v$ and $x$ in $\R^n$.
If $A\subset\R^n$ is measurable, we write $\vol(A)$ to denote
its volume, that is, its $n$-dimensional Lebesgue measure. 
As usual, we write $\GL(n)$ and $\SL(n)$ to  denote the general linear
and special linear groups in $\R^n$. The unit sphere of $\R^n$ will be denoted by $S^{n-1}$.

The \emph{difference body $DK$ of $K\in\K^n$} is the vector (or Minkowski) sum of $K$ and
its reflection in the origin, i.e.,
\begin{equation}\label{e: def DK}
DK := K + (-K).
\end{equation}

The {\it difference body inequality} or {\it Rogers-Shephard inequality} (see e.g. \cite{rogers.shephard}) constitutes the fundamental (affine) inequality relating the volume of the difference body $DK$ and the volume of $K$.
It is usually introduced together with a lower bound, which is a direct consequence of the Brunn-Minkowski inequality:

\noindent {\it Let $K\in\K^n$. Then
\begin{equation}\label{eq: RS and BM}
2^n \vol(K) \leq \vol(DK) \leq \binom{2n}{n} \vol(K).
\end{equation}
Equality holds on the left-hand side if and only if $K$ is centrally symmetric and on the right hand side precisely if $K$ is a simplex.}

We say that an operator $\di:\K^n\to\K^n$ satisfies a \emph{Rogers-Shephard type inequality} (in short RS) if there exists a constant $C>0$ such that for all $K\in\K^n$,
\begin{equation}\label{eq: RS}
\vol(\di K)\leq C \vol(K).
\end{equation}
Analogously, $\di$ satisfies a \emph{Brunn-Minkowski type inequality} (in short BM) if there exists a  constant $c>0$ such that for all $K\in\K^n$,
 \begin{equation}\label{eq: BM}
c \vol(K)\leq\vol(\di K).
\end{equation}

Sometimes we will consider other functionals than volume. In this case, we will say that $\di:\K^n\longrightarrow \K^n$ satisfies a \emph{Rogers-Shephard type inequality} (respectively, a Brunn-Minkowski type inequality) \emph{for the functional $\phi:\K^n\to\R$} if in \eqref{eq: RS} (resp. \eqref{eq: BM}) the volume is replaced by $\phi$.

As an operator on convex bodies $$\func{D}{\K^n}{\K^n}{K}{DK,}$$ the difference body enjoys several properties. It is continuous in the Hausdorff metric, $\SL(n)$-covariant and homogeneous of degree $1$. An operator $\di:\K^n\longrightarrow \K^n$ is said to be \emph{$G$-covariant} for a group of transformations $G$ if  for any $K\in\K^n$ it holds
\[
\di(g K)=g\di K \text{ for any } g\in \,G,
\]
and it is \emph{homogeneous of degree $k\in\R$} if for any $K\in\K^n$,
\[
\di(\lambda K)=\lambda^k \di K \text{ for any } \lambda> 0.
\]

Further, $K\mapsto DK$ is a translation invariant Minkowski valuation. Here an operator $\di$ is a \emph{Minkowski valuation} if 
for any $K,L\in \K^n$ with $K\cup L \in \K^n$,
\[
\di(K\cup L)+\di(K\cap L)=\di(K)+\di(L),
\]
where the addition on $\K^n$ is the Minkowski addition. An operator $\di$ is \emph{translation invariant} if 
\[
\di(K+t)=\di(K) \text{ for any } t\in\R^n.
\]

In fact, in \cite{ludwig} M. Ludwig proved that already continuity, translation invariance, Minkowski valuation and $\SL(n)$-covariance are enough to classify the difference body operator.

\begin{teor}[\cite{ludwig}]\label{t: ludwig class DK} Let $n\geq 2$. An operator $\di: \K^n\longrightarrow\K^n$ is a continuous, translation invariant and $\SL(n)$-covariant Minkowski valuation if and only if there is a $\lambda\geq 0$ such that $\di K=\lambda DK$.
\end{teor}

If the image of the operator $\di$ is restricted to centrally symmetric convex bodies, i.e., symmetric with respect to the origin, a characterization in the same direction is provided by R.\! Gardner, D.\! Hug and W.\! Weil in \cite{gardner.hug.weil1}. Following their notation let us introduce the notion of $o$-symmetrization.
Given a class of sets $\mathcal{C}$ and its subclass $\mathcal{C}_s$ of central symmetric (symmetric with respect to the origin) elements of $\mathcal{C}$, an operator $\di:\mathcal{C}\longrightarrow \mathcal{C}_s$ is called an {\it $o$-symmetrization}.

\begin{teor}[\cite{gardner.hug.weil1}]\label{t: GHW class DK}
Let $n\ge 2$. An $o$-symmetrization $\di:\K^n\to\K^n_s$ is continuous, translation invariant and $\GL(n)$-covariant if and only if there is a $\lambda\geq 0$ such that $\di K=\lambda DK$.
\end{teor}
We would like to notice that this result is obtained as a by-product of a systematic study of operations between convex sets (see \cite{gardner.hug.weil2} and \cite{milman.rotem} too).

However, none of the above classifications makes use of the fundamental affine isoperimetric inequalities attached to it, namely, \eqref{eq: RS and BM}.
In the spirit of these classifications of $DK$, we derive the following ones in which Rogers-Shephard inequality plays a prominent role.

\begin{theorem}\label{t:RS+cont+GL aK+b(-K)}Let $n\geq 2$. An operator $\di:\K^n\to\K^n$ is  continuous, $\GL(n)$-covariant and satisfies a Rogers-Shephard type inequality if and only if there are $a,b\geq 0$ such that 
 $\di K=aK+b(-K)$.
\end{theorem}

As a consequence of this we obtain the following corollary.

\begin{corollary}\label{cor:TI-RS}
Let $n\geq 2$. An $o$-symmetrization $\di:\K^n\to\K^n_s$ is
continuous, $\GL(n)$-covariant and satisfies a Rogers-Shephard type inequality if and only if there is a $\lambda\geq 0$ such that $\di K=\lambda DK$.
\end{corollary}

In this paper we also aim to study the interaction of continuity and $\GL(n)$-covariance with other properties which are usually attached to the difference body. Several results will be obtained in this direction in Section \ref{sec: GL}. Nevertheless, only when adding the translation invariance property a characterization result for the difference body operator is achieved.

\begin{theorem}\label{th: ti}
Let $n\geq 2$. An operator $\di:\K^n\to\K^n$ is continuous, $\GL(n)$-covariant and translation invariant if and only if there is a $\lambda\geq 0$ such that $\di K=\lambda D K$.
\end{theorem}

The role of a Brunn-Minkowski type inequality in classifying the difference body happens not to be relevant accompanied of $\GL(n)$-covariance and continuity: as we shall prove in Theorem \ref{BMinGL}, satisfying BM is a consequence of the join of both, $\GL(n)$-covariance and continuity. We provide several examples showing that without $\GL(n)$-covariance
and/or continuity, the difference body is far from being unique.

\smallskip
The paper has the following outline: in Section \ref{s: proj cov} we prove Theorem \ref{th: ti}. Further, in Section \ref{s: RS classif} we study the role of Rogers-Shephard inequality in the already mentioned (classical) classifications of the difference body and
provide the proof of Theorem \ref{t:RS+cont+GL aK+b(-K)}. In Section \ref{sec: GL} we investigate
the interplay of continuity and $\GL(n)$-covariance with other properties, such as monotonicity,  additivity or BM. Finally, we provide several
examples of operators which satisfy these properties and have no direct relation with the difference body, since they are not continuous or $\GL(n)$-covariant. 

\section{Projection covariance is a powerful tool: proof of Theorem \ref{th: ti}}\label{s: proj cov}

Following the ideas of \cite{gardner.hug.weil1}, in this section we will prove that projection covariance happens to be a very powerful 
assumption in order to classify an operator in the spirit described in the introduction. 
An operator $\di: \K^n\to \K^n$ is \emph{projection covariant} if for any $E\in\Gr(k,n)$, $1\leq k\leq n-1$ and any $K\in \K^n$, $$\di(K|E)=(\di K)|E,$$ where $A|E$ is the orthogonal projection of the set $A$ onto $E\in\Gr(k,n)$. As usual, $\Gr(k,n)$ is the Grassmannian of linear $k$-dimensional subspaces of $\R^n$.
We shall see that, in our context, projection covariance is equivalent to $\GL (n)$-covariance and continuity. 
Moreover, it provides us with an {\it almost} explicit description (see Theorem \ref{thm2s}) of the image of the operators $\di: \K^n\to \K^n$ sharing this property in terms of the support function.

Our main results rely strongly on slight variations of \cite[Lemma 7.4, Lemma 8.1 and Theorem 8.2]{gardner.hug.weil1}. 
We include (most of) the details of the proof for completeness.

\begin{theorem}\label{thm2s}
Let $n\ge 2$. The operator $\di:{\mathcal{K}}^n\rightarrow{\mathcal{K}}^n$ is projection covariant if and only if there is a planar convex body $M\subset\R^2$ such that
\begin{equation}\label{ns}
h(\di K,x)=h_{M}\left(h_K(x),h_{-K}(x)\right),
\end{equation}
for all $K\in {\mathcal{K}}^n$ and all $x\in\R^n$.
\end{theorem}

If the operator $\di$ satisfies the above hypothesis, we will say that the planar convex body $M$ obtained by this result is the \emph{associated body to $\di$}. However, it is known that $M$ needs not to be unique (see e.g. \cite[Section 8]{gardner.hug.weil1}).

\begin{proof}
We first show that there is a homogeneous of degree 1 function $f:H\to\R$ such that
\begin{equation}\label{441s}
h(\di K,x)=f\left(h_{K}(x),h_{-K}(x)\right),\quad\forall x\in\R^n,
\end{equation}
with $H=\{(s,t)\in\R^2: -s\le t\}$.

Let $u\in S^{n-1}$ and $l_u=\lin{u}$. Since $\di:{\mathcal{K}}^n\rightarrow{\mathcal{K}}^n$ is projection covariant, we have
\begin{equation}\label{eqbs}
(\di K)|\,l_u=\di(K|\,l_u).
\end{equation}
If $I$ is a closed interval in $l_u$, then, $\di I\subset l_u$. Thus we can find functions $f_u:H\to\R$, $g_u:H\to\R$,
such that
\begin{equation}\label{eqn1s}
\di [-su,tu]=[-f_u(s,t)u,g_u(s,t)u],
\end{equation}
and $-f_u(s,t)\leq g_u(s,t)$ whenever $(s,t)\in H$. 

We let $0\le \alpha\le 1$ and choose $v\in S^{n-1}$ such that $\langle u, v\rangle=\alpha$. Applying (\ref{eqbs}) with $K=[-su,tu]$ and $l_v$ instead of $l_u$, and \eqref{eqn1s}, we obtain
\begin{align*}
(\di[-su,tu])|\,l_v&=[-f_{u}(s,t)u,g_{u}(s,t)u]|\,l_v=
\alpha[-f_{u}(s,t)v,g_{u}(s,t)v]
\\
&=\di([-su,tu]|\,l_v)=\di[-\alpha sv,\alpha tv]=
[-f_{v}(\alpha s,\alpha t)v,g_{v}(\alpha s,\alpha t)v],
\end{align*}
for all $(s,t)\in H$.  Therefore $f_v(\alpha s,\alpha t)=\alpha f_u(s,t)$ and  $g_v(\alpha s,\alpha t)=\alpha g_u(s,t)$, for all $(s,t)\in H$. We also obtain 
\begin{equation}\label{alphav}
f_u(\alpha s,\alpha t)=\alpha f_v(s,t)
\end{equation}
by interchanging $u$ and $v$, and hence
$$
f_u(\alpha^2 s,\alpha^2 t)=\alpha f_v(\alpha s,\alpha t)=\alpha^2f_u(s,t),$$
for all $(s,t)\in H$.  

Taking $r=\alpha^2$, we have $f_u(rs,rt)=rf_u(s,t),$
for $0\le r\le 1$ and $s,t\in H$.  Replacing $s$ and $t$ by $s/r$ and $t/r$, respectively, it follows that
$f_u$ is homogeneous of degree 1.

Next, we need to prove that for any $u,v\in S^{n-1}$, $f_u=f_v$ and $g_u=g_v$.  First we fix $u\in S^{n-1}$ and let $v\in S^{n-1}$ be such that $\langle u, v\rangle>0$. We put $\alpha=\langle u, v\rangle$.  Using (\ref{alphav}) and the homogeneity of $f_u$, we obtain
$$\alpha f_v(s,t)=f_u(\alpha s,\alpha t)=\alpha f_u(s,t),$$
for all $s,t\ge 0$.  This shows that $f_v=f_u$ for all such $v$ and consequently $f_u=f$ is independent of $u$.
Applying the same arguments to $g_u$ we obtain $g_u=g$ and that it is also homogeneous of degree 1.

Joining (\ref{eqbs}) and (\ref{eqn1s}) and using that $f_u=f$ and $g_u=g$, we obtain
\begin{align*}
(\di K)|\,l_u&=[-h_{\di K}(-u)u,h_{\di K}(u)u]
\\&=\di (K|\,l_u)=\di[-h_{K}(-u)u,h_{K}(u)u]
=[-f\left(h_{-K}(u),h_K(u)\right)u,g\left(h_{-K}(u),h_K(u)\right)u],
\end{align*}
for all $u\in S^{n-1}$. Thus,
\begin{equation}\label{hdknew}
h_{\di K}(-u)=f\left(h_{-K}(u),h_{K}(u)\right),\quad
h_{\di K}(u)=g\left(h_{-K}(u),h_{K}(u)\right),\end{equation}
for all $u\in S^{n-1}$, which implies $f(s,t)=g(t,s)$ for $(s,t)\in H$. Now, it is enough to observe that $f$ and $g$ determine $\di K$ for any
$u\in S^{n-1}$ and 
using the homogeneity of $f$ we obtain (\ref{441s}) for any $u\in\R^n$.

Next we prove that $f:H\to\R$ is a support function. Notice that with \eqref{441s} we have obtained, as in \cite[Lemma 8.1]{gardner.hug.weil1}, a homogeneous of degree $1$ function $f$. Thus, we can use Theorem 8.2 in \cite{gardner.hug.weil1} to obtain the convexity of $f$.

Indeed, considering the convex body $K_0=[-e_2,e_1]$, the segment joining the points $-e_2$ with $e_1$, and the projection covariance of $\di$, one obtains that $M=\di K_0$ is a feasible body associated to $\di$, using the properties of $f$.

The converse of the result follows directly from the properties of support functions.
\end{proof}

Centrally symmetric convex bodies play a special role when dealing with the difference body, as they are the only {\it fixed points} (up to maybe a constant) of the operator.
From now on and for the sake of brevity, we will write \emph{$o$-symmetric convex sets} for origin or centrally symmetric sets.

\begin{remark}\label{r: proj cov identity property}
From the above result it immediately follows that if $K$ is $o$-symmetric, then
$h(\di K,u)=h_M(1,1)h(K,u)$ and $h_M(1,1)\geq 0$. 
\end{remark}

In \cite{gardner.hug.weil1} it was proved that an $o$-symmetrization enjoys continuity and $\GL(n)$-covariance if and only if it is projection covariant \cite[Lemma 4.3 and Corollary 8.3]{gardner.hug.weil1}. The same proof works for a non-necessarily $o$-symmetric operator $\di:\K^n\to\K^n$. Hence, we have the following lemma.

\begin{lemma}[Corollary 8.3 in \cite{gardner.hug.weil1}]\label{diCGimpliesP}
The operator $\di:{\mathcal{K}}^n\rightarrow {\mathcal{K}}^n$ is continuous and $\GL(n)$-covariant if and only if it is projection covariant.
\end{lemma}

As a consequence of this result based on the projection covariance, implied by the $\GL(n)$-covariance and continuity, Theorem \ref{th: ti} follows, in the same way as Theorem B was proved in \cite{gardner.hug.weil1}. We include again (most of) the details for completeness.

\begin{proof}[Proof of Theorem \ref{th: ti}]
For $K\in {\mathcal{K}}^n$, it has to be shown $h(\di K,x)=\lambda h_{D K}(x)$ for any $x\in\R^n$. Since $\di$ is invariant under
translations, it is enough to find a translation $t$ for which $h(\di (K+t),x)=\lambda h_{D K}(x)$.  The vector $t\in\R^n$ satisfying
$\langle x, t\rangle=\frac{1}{2}(h(-K,x)-h(K,x))$ fulfills the stated property.
\end{proof}

\section{Rogers-Shephard inequality as classifying property}\label{s: RS classif}

In this section we show that if the condition of translation invariance in Theorem \ref{th: ti} is replaced by the condition RS -of satisfying a Rogers-Shephard type inequality (cf.\! \eqref{eq: RS})-, then other operators than the difference body (see Theorem \ref{t:RS+cont+GL aK+b(-K)}) appear, but they are not far from it. 

If instead, we replace, now in Theorem \ref{t: GHW class DK}, translation invariance by RS (notice that in this case the assumption of $o$-symmetrization is present) we do actually obtain the difference body (cf. Corollary \ref{cor:TI-RS}).

Now we aim to prove Theorem \ref{t:RS+cont+GL aK+b(-K)}, i.e., that a continuous and $\GL(n)$-covariant operator $\di:\K^n\to\K^n$ satisfying RS is 
necessarily of the form $\di K=aK+b(-K)$ for $a,b\geq 0$.
We will need the following lemmas for the proof.

\begin{lemma}\label{hM11=0}
Let $n\geq 2$ and $\di:\K^n\to\K^n$ be continuous and $\GL(n)$-covariant with $M\in\K^2$ as associated convex body. If $h_M(1,1)=0$, then $\di K=\{0\}$ for every $K\in\K^n$.
\end{lemma}
\begin{proof}
We show first that $h_M(a,b)\leq 0$ for every $a,b\geq 0$. Let $K$ be a convex body with $h_K(e_1)=h_K(-e_1)\neq 0$, $h_K(e_2)=h_K(-e_2)\neq 0$
but $a:=h_K(e_1+e_2)\neq h_K(-e_1-e_2)=:b$. Note that for every $a,b>0$, there exists a convex body satisfying these conditions. 
Since $\di K$ is a convex body, it holds, as claimed, 
\[
\begin{split}
h_M(a,b) &=  h_M(h_K(e_1+e_2),h_K(-e_1-e_2)) \\ &\leq h_M(h_K(e_1),h_K(-e_1))+h_M(h_K(e_2),h_K(-e_2))=0.
\end{split}
\]

Next we show that $M$ is necessarily a segment of the form $[(-r,-r),(-s,-s)]$, for some $0\leq s\leq r$. In order to do it, it is enough to prove that $h_M(1,-1)\leq 0$ and $h_M(-1,1)\leq 0$. Let us consider convex bodies $K,L$ satisfying $h_K(e_1+e_2)=-1$, $h_K(-e_1-e_2)=1$, $h_L(e_1+e_2)=1$, $h_L(-e_1-e_2)=-1$ and $h_K(\pm e_1),h_K(\pm e_2), h_L(\pm e_1), h_L(\pm e_2)> 0$. Since $\di K$ and $\di L$ are convex we obtain, by using the previous claim, that $h_M(-1,1)\leq 0$ and $h_M(1,-1)\leq 0$. 

Hence, for any convex body $K\in\K^n$, we have $h(\di K,u)=h_M(h_K(u),h_K(-u))=\max\{-s,-r\}(h_K(u)+h_K(-u))=h(\di K,-u).$ Since $h(\di K,u)+h(\di K,-u)\geq 0$, we obtain $r=s=0$, and hence, $M=\{(0,0)\}$. 
\end{proof}

\begin{remark}\label{converse_hM11=0} We notice, that from the above proof it follows that
\begin{enumerate}
\item the converse of Lemma \ref{hM11=0} holds. Indeed, the image of a full dimensional $o$-symme\-tric convex body is the origin if and only if $h_M(1,1)=0$ (see Remark \ref{r: proj cov identity property}). 
\item if $\di:\K^n\to\K^n$ is a continuous $\GL(n)$-covariant operator and the image of some $o$-symmetric convex body (not 0-dim) is a point, then $\di\equiv 0$.
\end{enumerate}
\end{remark}

The next lemma will be very useful to prove most of the following results.

We will say that an operator $\di:\K^n\to\K^n$ is \emph{trivial}, if its image consists only of the origin.

\begin{lemma}\label{width ; width_comp} Let $n\geq 2$ and $\omega (K,u)=h(K,u)+h(K,-u)$ denote the width of $K\in\K^n$ in the direction $u\in S^{n-1}$.
\begin{enumerate}
\item If $\di:\K^n\to\K^n$ is a continuous and $\GL(n)$-covariant operator with $M\in\K^2$ as associated convex body, then 
$$\omega(K,u)h_M(1,1)\leq\omega(\di K,u),\quad\forall u\in S^{n-1}.$$

\item If $\di:\K^n\to\K^n$ is a continuous, non-trivial, $\GL(n)$-covariant operator, then $\aff K\subseteq \aff \di K$. In particular,
$\di$ maps convex sets of dimension $n$ into convex sets of dimension $n$.

\item If $K,L\in\K^n$ satisfy $\omega(K,u)\leq\omega(L,u)$ for every $u\in S^{n-1}$, then
$$\vol(K)\leq 2^{-n}\binom{2n}{n}\vol(L).$$
\end{enumerate}
\end{lemma}

\begin{proof}\hfill
\begin{enumerate}
\item It follows directly from the properties of the support function of $M$, that
\[
\begin{split}
\omega(\di K,u)= & h_M(h_K(u),h_K(-u))+h_M(h_K(-u),h_K(u))\\ & \geq (h_K(u)+h_K(-u))h_M(1,1)=\omega(K,u)h_M(1,1).
\end{split}
\]

\item It follows directly from the first item and Remark \ref{converse_hM11=0}.

\item From the hypothesis, it follows that $h(K+(-K),u)\leq h(L+(-L),u)$ for every $u\in S^{n-1}$, thus, $DK\subseteq DL$. Now, using \eqref{eq: RS and BM} we obtain, as stated,
$$\vol(K)\leq 2^{-n}\vol(DK)\leq 2^{-n}\vol(DL)\leq 2^{-n}\binom{2n}{n}\vol(L).$$
\end{enumerate}
\end{proof}

The following remark will be used in the following. We include it for future reference.
\begin{remark}\label{r: M segment x=y}
Let $\di$ be a continuous and $\GL(n)$-covariant operator and $M$ its associated body.
If $M$ is a segment on the line $\{(x,y)\,:\,x=y\}\subset\R^2$, then $\di K= \lambda DK$ for certain $\lambda\geq 0$.
\end{remark}

The next result states that if the continuous and $\GL(n)$-covariant operator $\di$ preserves dimensions, then it is not far from being the difference body operator.

\begin{proposition}\label{equalDim}
Let $n\geq 2$ and $0\leq k\leq n-1$. An operator $\di:\K^n\to\K^n$ is continuous, $\GL(n)$-covariant and satisfies $\dim K=\dim\di K$ for any $K\in\K^n$ of $\dim K=k$ if and only if there are $a,b\geq 0$ such that $\di K=aK+b(-K)$.
\end{proposition}
\begin{proof}
Let $u\in S^{n-1}$ and $K\in\K^n$ with $\dim K=k$ be such that $\omega(K,u)=0$ and $h(K,u)>0$. Then, since $\dim K=\dim\di K$, using Lemma \ref{width ; width_comp} we obtain that $\omega(\di K,u)=0$. 
From $h(K,u)=-h(K,-u)$, it follows
\begin{align*}\omega(\di K,u)&=h_M(h_K(u),h_K(-u))+h_M(h_K(-u),h_K(u))\\&=h_K(u)(h_M(1,-1)+h_M(-1,1))\\&=h_K(u)\omega(M,(1,-1)).
\end{align*}

If $\omega(M,(1,-1))>0$,  
according to the above calculation, $\omega(\di K,u)>0$. Hence, it holds $\omega(M,(1,-1))=0$. 
The latter implies that $M$ is a (possibly degenerated) segment lying in a line parallel to $\{x=y\}$.
Let $M=[(a,a),(b,b)]+(p,0)$ with $a,b,p\in\R$ and $a\leq b$. Then, using Remark \ref{r: M segment x=y} we obtain
\begin{align*}
h_{\di K}(x)&=h_{M}(h_K(x),h_{-K}(x))=h_{[(a,a),(b,b)]}(h_K(x),h_{-K}(x))+ph_K(x)
\\&=\lambda h_{K+(-K)}(x)+ph_K(x)=h_{(\lambda +p)K+\lambda(-K)}(x).
\end{align*}
The converse is clear. 
\end{proof}

We notice that for the proof it is enough to assume only, that for some $K\in\K^n$ with $0\leq \dim K \leq n-1$ there is an appropriate translate
of it, $K+t$ such that $\dim (K+t)=\dim \di (K+t)$.

We would also like to remark the following consequence:
\begin{corollary}
Let $n\geq 2$. An operator $\di:\K^n\to\K^n$ is continuous, $\GL(n)$-covariant and satisfies $\di\{p\}=\{0\}$ for some $p\neq 0$ if and only if there is a $\lambda\geq 0$ such that $\di K=\lambda DK$.
\end{corollary}

As a consequence of the above proof we observe the following remark.

\begin{remark}\label{r: w for  lower dim}
Let $\di:\K^n\to\K^n$ be a continuous and $\GL(n)$-covariant operator and $M$ the body associated to it. It holds, 
\begin{enumerate}
\item if $K\in\K^n$ with $0\leq\dim K\leq n-1$, then $\omega(\di K,u)=h_K(u)\omega(M,(1,-1))$ for every $u\in S^{n-1}$ such that $\omega(K,u)=0$, 
\item $\omega(M,(1,-1))\neq 0$ if and only if $\dim K <\dim \di K$ for some $K\in \K^n$ of $\dim K\leq n-1$.
\end{enumerate}
\end{remark}

Now we are ready to prove Theorem \ref{t:RS+cont+GL aK+b(-K)}.
\begin{proof}[Proof of Theorem \ref{t:RS+cont+GL aK+b(-K)}]
First, we recall that under the assumption of RS, if $K\in \K^n$ with $\dim K\leq n-1$ then we necessarily have $\dim \di K\leq n-1$. 
Thus, from Lemma \ref{width ; width_comp}.i and Proposition \ref{equalDim} with $k=n-1$, we can assert that $\omega(M,(1,-1))=0$ and we have, as in Proposition \ref{equalDim}, that $\di K=aK+b(-K)$ for some $a,b\geq 0$. 

It is clear, that each operator of the form $K\mapsto aK+b(-K)$, $a,b\geq 0$ satisfies a Rogers-Shephard inequality. 
\end{proof}

The next corollary indicates that asking for RS is as strong as translation invariance if $o$-symmetrization
is assumed, as it is the case of Theorem \ref{t: GHW class DK}.

\begin{corollary}\label{cor:TI-RS SL}
Let $n\geq 2$. An $o$-symmetrization $\di:\K^n\to\K^n_s$ is continuous, $\GL(n)$-covariant and satisfies RS
if and only if there exists a $\lambda\geq 0$ such that $\di K=\lambda DK$.
\end{corollary}

 \begin{proof}
From Theorem \ref{t:RS+cont+GL aK+b(-K)}, there exist $a,b\geq 0$ such that $\di K=a K+ b(-K)$ for $K\in\K^n$.
We have to prove that $a=b$. In order to do so, it is enough to consider $K\in \K^n$ with $h(K,-u)=0$ and $h(K,u)>0$, 
which an appropriate translation of any full dimensional ($\dim K=n$) convex body does.
Hence, since $K$ satisfies $h(\di K,u)=h(a K+ b(-K),u)=h(a K+ b(-K),-u)=h(\di K,-u)$, we easily obtain $a=b$.
 \end{proof}

\section{Interaction of other properties with projection covariance}\label{sec: GL}

In the previous sections we have seen that continuity and $\GL(n)$-covariance (or equivalently projection covariance) 
play a strong classifying role for the difference body, being however, on their own, not enough for it.
In this section we gather together some aspects of the interplay of continuity and $\GL(n)$-covariance
with further assumptions such as Brunn-Minkowski inequality, monotonicity or additivity.
We also provide (further) examples showing that removing either $\GL(n)$-covariance or continuity forces us to assume several additional assumptions
in order to get some close to the difference body. 

We start proving that continuity and $\GL(n)$-covariance imply the homothety property. 
We say that a non-trivial operator $\di:\K^n\to\K^n$ satisfies the \emph{homothety property} if there exists some $\lambda>0$ such that 
$\di K=\lambda K$ for every $K\in\K^n_s$. If $\lambda=1$ one gets the \emph{identity property} used in \cite{gardner.hug.weil1}.

\begin{proposition}\label{r: SL cov takes sym to sym}
Let $n\geq 2$. If $\di:\K^n\to\K^n$  is a non-trivial, continuous, $\GL(n)$-covariant operator, then it has the homothety property.
\end{proposition}

\begin{proof}
From Theorem \ref{thm2s} we know that $h(\di K,u)=h_M(h_K(u),h_K(-u))$.
Thus, if $K=-K$, we immediately obtain $h(\di K,u)=h_M(1,1)h(K,u)$ from which the result follows.
\end{proof}

\begin{remark}
We notice that an operator having the homothety property needs not be either continuous or $\GL(n)$-covariant:
\[
\di K = \left\{\begin{array}{lr} 
 K, & \text{  if  }  K =-K\\
 B_n, & \text{  otherwise  }  
\end{array}\right.
\]
where $B_n$ is the unit Euclidean ball.
\end{remark}

Next we prove that if a non-trivial operator is continuous and $\GL(n)$-covariant, then it satisfies BM. Indeed even more can be proven:
such an operator satisfies a Brunn-Minkowski type inequality for every quermassintegral.

We recall that the {\it Steiner formula} states that the volume of the Minkowski sum of a convex body $K$ 
and a non-negative dilation of $B_n$ is a polynomial 
\[
\vol(K+\rho B_n)=\sum_{i=0}^{n}\rho^i\binom{n}{i}W_i(K),
\]
whose coefficients $W_i(K)$, up to normalization, are
the so-called \emph{quermassintegrals of $K$}. From the above it follows that $\W_0(K)=\vol(K)$ and $W_n(K)=\vol(B_n)$ for all $K$.
Indeed, this is a particular case of a much more general fact (see \cite{schneider.book14}): for convex bodies $K_1,\dots,K_m$ and $\lambda_1,\dots,\lambda_m\in\R$ non-negative, the volume of the linear combination $\lambda_1 K_1+\dots +\lambda_m K_m$ is a homogeneous polynomial, whose coefficients are the so-called \emph{mixed volumes} (of whose quermassintegrals are particular cases), namely
\begin{equation}\label{eq: mixed vol}
\vol(\lambda K_1+\dots +\lambda K_m)=\sum_{1\leq i_1,\dots,i_n\leq m}\lambda_{i_1}\dots\lambda_{i_n}V(K_{i_1},\dots,K_{i_n}).
\end{equation}

\begin{theorem}\label{BMinGL}
Let $n\geq 2$. If $\di:\K^n\to\K^n$ is a non-trivial, continuous and $\GL(n)$-covariant operator, then it satisfies a Brunn-Minkowski type inequality for every quermassintegral $\W_i$, $i=0,\dots,n-1$, i.e., there is a constant $c_{n,i,\di}>0$ such that 
$$W_i(K)\leq c_{n,i,\di}W_i(\di K).$$
\end{theorem}

\begin{proof}
Let $M$ be the associated convex body to $\di$, ensured by Lemma \ref{diCGimpliesP} and Theorem \ref{thm2s}. Since $\di$ is non-trivial, $h_M(1,1)>0$ by Lemma \ref{hM11=0} and Remark \ref{r: proj cov identity property}.

We will prove first the statement for $i=0$, that is, for the volume. Let $K\in\K^n$.
From Lemma \ref{width ; width_comp}.i, we know that
$$\omega(h_M(1,1)K,u)\leq \omega(\di K,u),$$
for all $u\in S^{n-1}.$

Using now Lemma \ref{width ; width_comp}.iii, we have 
$$\vol(K)\leq (2(h_M(1,1))^{-n}\binom{2n}{n}\vol(\di K),$$
and hence, a Brunn-Minkowski type inequality for $\di$ and $\W_0=\vol$.

We notice that $M$ is associated to the operator $\di$ and not to some convex body, hence, the constant is independent of the convex body.

In order to prove the statement for $i=1,\dots,n-1$ we will make use of the so-called Kubota integral recursion formula for quermassintegrals, which states (see e.g. \cite[(5.72)]{schneider.book14}) 
\begin{equation}\label{kubota}
W_{n-l}(\di K)= c(l,n)\int_{E\in \Gr(l,n)}\vol_l((\di K)|E)dE,
\end{equation}
where $dE$ denotes the probability measure on $\Gr(l,n)$ and $c(l,n)$ is a constant depending only on $n,l$, $0\leq l \leq n$.

From Lemmas \ref{width ; width_comp}.i and \ref{diCGimpliesP}, we have
$$
\omega(h_M(1,1)(K|E),u)\leq\omega(\di(K|E),u)=\omega((\di K)|E, u)
$$
for any $E\in\Gr(l,n)$ and $u\in S^{n-1}$.

Lemma \ref{width ; width_comp}.iii applied on $E\in \Gr(l,n)$ and the projection covariance yield
$$
\vol_l(h_M(1,1)(K|E))\leq 2^{-l}\binom{2l}{l}\vol_l((\di K)|E).
$$
Using the homogeneity of the volume and plugging the latter inequality in
\eqref{kubota}  
we obtain the result.
\end{proof}

Using Proposition \ref{equalDim} for an $(n-l-1)$-dimensional convex body,
we obtain that the condition of satisfying a Rogers-Shephard for a  quermassintegrals {\it almost characterizes} the difference body operator, that is, the following statement holds: 

\begin{theorem}
Let $n\geq 2$. An operator $\di:\K^n\to\K^n$ is continuous $\GL(n)$-covariant and satisfies a Rogers-Shephard type inequality for some $l$-th quermassintegral, $l\in\{0,1,\dots,n-1\}$ if and only if there are $a,b\geq 0$ such that $\di K=aK+b(-K)$.
\end{theorem}

\begin{proof}First we note that for $l=0$, the statement coincides with Theorem \ref{t:RS+cont+GL aK+b(-K)}. 

Let $M$ be the associated body to the operator $\di$ such that $W_l(\di K)\leq c W_l(K)$ for some $c>0$ and a fixed $l\in\{1,\dots,n-1\}$. 

Lemma \ref{width ; width_comp} yields $\dim K\leq\dim\di K$ for any $K\in \K^n$. 
Let $K\in\K^n$ be such that $\dim K=n-l-1$, so that $W_{l}(K)=0$. 

Since $1 \leq l\leq n-1$, we can find $u\in S^{n-1}$ such that $\omega(K,u)=0$. Then, from Remark~\ref{r: w for  lower dim} we have that $\omega(\di K, u)=h(K,u)\omega(M,(1,-1))$. 
If $\omega(K,(1,-1))>0$, just using an appropriate translation of $K$ instead of $K$, if necessary, it holds that there exists an $(n-l-1)$-dimensional convex $K$ with $\dim \di K\geq n-l$. Hence, $\di$ does not satisfies a Rogers-Shephard inequality for $W_{l}$. Thus, $\omega(M,(1,-1))=0$ which, as in the proof of Proposition \ref{equalDim} yields that $\di K=aK+b(-K)$ for some $a,b\geq 0$. 

We notice, that each operator $K\mapsto aK+b(-K)$, $a,b\geq 0$ satisfies a Rogers-Shephard inequality for $W_l$, $1\leq l\leq n-1$, using that
$-K\subseteq n K$ for $K\in\K^n$ having centroid at the origin. The latter yields the converse.
\end{proof}

Monotonicity happens to be, as BM, a property shared by any continuous and $\GL(n)$-covariant operator.
In order to prove it, we will make use of the so called $M$-sum.

Given a subset $M$ of $\R^2$ and sets $K,L\in\R^n$, the {\it $M$-sum} of $K,L$ is the set
$$
K +_M L=\{ax+by\,:\,x\in K,y\in L, (a,b)\in M\}.
$$
If $M\in\K^2$ and $K,L\in\K^n$, then it is not difficult to check, that
$$
h(K+_M L,u)=h_M(h_K(u),h_L(u))
$$
for $u\in S^{n-1}$. For further details we refer to \cite{gardner.hug.weil1}.
We notice, that the $M$-sum was already (implicitly) used in Theorem \ref{thm2s}.

\begin{proposition}\label{BM then monotonic}
Let $n\geq 2$. If $\di:\K^n\to\K^n$ is a continuous and $\GL(n)$-covariant operator, then it is monotonic.
\end{proposition}
\begin{proof}
From Theorem \ref{thm2s} we know that 
$$\di K=K +_M (-K)=\bigcup_{(a,b)\in M} aK+b(-K).$$ 
It follows directly from $K\subseteq L$, that
$\di K \subseteq \di L$.
\end{proof}

In the next result we impose additivity to the continuous and $\GL(n)$-covariant operator. We recall that an operator $\di:\K^n\to\K^n$ is said to be \emph{additive} if for any $K,L\in \K^n$ it satisfies that
$\di(K+L)=\di K +\di L$.

\begin{proposition}
Let $n\geq 2$. An operator $\di:\K^n\to\K^n$ is continuous, $\GL(n)$-covariant and additive if and only if there are $a,b\geq 0$ such that $\di K=aK+b(-K)$.
\end{proposition}
\begin{proof}Let $u\in S^{n-1}$, $c>0$ and $K, L\in\K^n$ such that $(h_K(u),h_K(-u))=(c,-c)$ and $(h_L(u),h_L(-u))=(-c,c)$. Then,
$$h(\di(K+L),u)=h_M(c-c,-c+c)=0 \text{  and  }$$
$$h(\di K+\di L,u)=h_M(c,-c)+h_M(-c,c)=c\, \omega(M,(1,-1)).$$
Thus, $\di$ is additive if and only if $\omega(M,(1,-1))=0$, that is, if and only if $\di K=aK+b(-K)$ for some $a,b\geq 0$ (cf. Remark \ref{r: M segment x=y}).
\end{proof}

We have seen above, that the assumptions of continuity and $\GL(n)$-covariance, Rogers-Shephard inequality {\it almost} characterize the difference body (see Theorem \ref{t:RS+cont+GL aK+b(-K)}).
 
A similar phenomenon happens when to continuity and $\GL(n)$-covariance is added the hypothesis of Minkowski valuation.  
In \cite{wannerer.equiv}, T. Wannerer proved that if we omit the assumption of translation invariance in Ludwig's classification of the difference body (Theorem \ref{t: ludwig class DK}), we are actually not that far from the difference body.

\begin{teor}[\cite{wannerer.equiv}]\label{t: class Wannerer}Let $n\geq 3$. An operator $\di:\K^n\to\K^n$ is a continous, $\GL(n)$-covariant Minkowski valuation if and only if there are $a,b,c,d\geq 0$ such that 
\begin{equation}\label{eq:GL-MVal}
\di K=aK+b(-K)+c\conv(\{0\}\cup K)+d\conv(\{0\}\cup (-K)).
\end{equation}
\end{teor}

Let us notice that the continuous, $\GL(n)$-covariant operator $\di K=\conv(\{0\}\cup K)$ can be obtained as
$$h(\di K,u)=\max\{0,h(K,u)\}=h_{[(0,0),(1,0)]}(h_K(u),h_K(-u)),$$
i.e., a body $M$, associated to $\di$ (cf. Theorem \ref{thm2s}) can be chosen to be the segment joining the origin and the point $(1,0)$. 

Hence, the body $M:=\{(a,0)\}+\{(0,b)\}+c[(0,0),(1,0)]+d[(0,0),(0,1)]$ is an associated body to \eqref{eq:GL-MVal}. In other words, for any $u\in S^{n-1}$, $h(\di K,u)$ is given by
\begin{equation}\label{eq:hM-MVal}
h_M(h_K(u),h_K(-u))=\left\{\begin{array}{ll}(a+c)h_K(u)+(b+d)h_K(-u), &\textrm{ if }h_K(u),h_K(-u)\geq 0
\\ (a+c)h_K(u)+bh_K(-u), &\textrm{ if }h_K(u)\geq 0,\, h_K(-u)< 0
\\ ah_K(u)+(b+d)h_K(-u), &\textrm{ if }h_K(u)<0,\, h_K(-u)\geq 0.
\end{array}\right.
\end{equation}

Using Theorem \ref{thm2s}, we prove that Theorem \ref{t: class Wannerer} also holds for $n=2$ and give a new proof of it. 
\begin{theorem}\label{t: Wannerer n=2}
Let $n\geq 2$. An operator $\di:\K^n\to\K^n$ is a continous, $\GL(n)$-covariant Minkowski valuation if and only if there are $a,b,c,d\geq 0$ such that
\begin{equation}\label{eq:GL-MVal 2}
\di K=aK+b(-K)+c\conv(\{0\}\cup K)+d\conv(\{0\}\cup (-K)).
\end{equation}
\end{theorem}
\begin{proof}Let $M$ be an associated convex body to $\di$, $u\in S^{n-1}$ and $\alpha,\beta>0$. 
Choose convex bodies $K,L$ such that 
$$h(K,u)=h(L,-u)=h(K\cap L,u)=h(K\cap L,-u)=0,$$ 
$$h(L,u)=h(K\cup L,u)=\alpha, \quad h(K,-u)=h(K\cup L,-u)=\beta.$$
Since $\di:\K^n\to \K^n$ is a Minkowski valuation, it holds
$$h_M(\alpha,\beta)+h_M(0,0)=h_M(0,\beta)+h_M(\alpha,0),$$
that is,
$$h_M(\alpha,\beta)=\alpha h_M(1,0)+\beta h_M(0,1),$$
and $h_M(\alpha,\beta)$ is a linear function when $\alpha,\beta\geq 0$. 

Next, by choosing $K'$ and $L'$ such that 
$$h(K',u)=h(K'\cup L',u)=0,\quad h(L',-u)=h(K'\cup L',-u)=\beta,$$ 
$$h(K',-u)=h(K'\cap L',-u)=-h(L',u)=-h(K'\cap L',u)=\alpha,$$
it similarly follows $h_M(-\alpha, \beta)
=\alpha(h_M(0,1)-h_M(-1,1))+\beta h_M(0,1)$.

Finally, by choosing $K''$ and $L''$ such that 
$$h(K'',-u)=h(K''\cup L'',-u)=0,\quad h(L'',u)=h(K''\cup L'',u)=\beta,$$ 
$$h(K'',u)=h(K''\cap L'',u)=-h(L'',-u)=-h(K''\cap L'',-u)=\alpha,$$
it follows $h_M(\beta, -\alpha)=\beta h_M(1,0)+\alpha(h_M(-1,1)-h_M(1,0))$.

The three conditions obtained for the support function of $h_M$ imply that \eqref{eq:hM-MVal} is satisfied considering $a=h_M(0,1)-h_M(-1,1)$, $b=h_M(-1,1)-h_M(1,0)$, $a+c=h_M(1,0)$ and $b+d=h_M(0,1)$. Thus, $\di K$ is given as in \eqref{eq:GL-MVal 2}.

Since each of the summands of \eqref{eq:GL-MVal 2} defines a continuous, $\GL(n)$-covariant operator which is a Minkowski valuation, the result follows. 
\end{proof}

Using Theorem \ref{t: class Wannerer} and Theorem \ref{t: Wannerer n=2} we obtain the following result for $o$-symmetrizations:
\begin{corollary}Let $n\geq 2$. An $o$-symmetrization $\di:\K^n\to\K^n_s$ is a continuous, $\GL(n)$-covariant Minkowski valuation if and only if there are $a,b\geq 0$ such that
$$\di K=aDK+bD(\conv(\{0\}\cup K)).$$ 
\end{corollary}
\begin{proof}Let $u\in S^{n-1}$ and $K\in\K^n$ such that $h(K,u)>0$ and $h(K,-u)=0$. Theorem \ref{t: Wannerer n=2} ensures that
\[
\begin{split}
h(\di K,u)=& ah(K,u)+bh(K,-u)+c\max\{0,h(K,u)\}+d\max\{0,h(K,-u)\}\\ =&(a+c)h(K,u).
\end{split}
\]
On the other hand, using now $-u$ we have 
$$h(\di K,-u)=(b+d)h(K,u).$$

Since $\di K$ is $o$-symmetric, it necessarily holds $a+c=b+d$. 
Now choosing a convex body $K$ such that for $u\in S^{n-1}$ we have $h(K,u)>0$ and 
$h(K,-u)<0$, we will have that $h(\di K,u)=h(\di K,-u)$ if and only if the following equality holds:
$$(a+c)h(K,u)+bh(K,-u)=(b+d)h(K,u)+ah(K,-u).$$
The latter together with the already obtained relation for $a,b,c,d$ implies directly that $a=b$ and $c=d$, which proves the result.
\end{proof}

Finally we would like to remark that under $\GL(n)$-covariance and continuity, an operator is an $o$-symmetrization if and only if it is even. Indeed, since $-\mathrm{Id}\in\GL(n)$, we directly obtain, that $\di(-K)=-\di K$.

In the next subsection we will present some examples showing that a {\it random choice} of the properties we have been dealing with, does not, in general, get any close to the difference body.

\subsection{Examples of operators sharing properties with $DK$, but far from being it.}\label{ss: examples}

\medskip

First we would like to understand the role of RS in Theorems \ref{t: ludwig class DK} and \ref{t: GHW class DK}.
Since we have Theorem \ref{th: ti} and Theorem \ref{t: ludwig class DK}, it would only make sense to either replace one of the hypothesis in the latter by RS, or weaken any of them and add RS. Let us also recall, that replacing translation invariance by RS in both, Theorem \ref{t: ludwig class DK} and \ref{t: GHW class DK} has already been addressed in Theorem \ref{t:RS+cont+GL aK+b(-K)} and Corollary \ref{cor:TI-RS}.

Unfortunately, we are not aware of a non-continuous valuation which is $\GL(n)$-covariant, 
translation invariant and satisfies a Rogers-Shephard inequality. 
However, if any of the other assumptions in Theorems \ref{t: ludwig class DK} and \ref{t: GHW class DK} is replaced by Rogers-Shephard inequality, there is, in general, no possibility of getting close to a characterization of the difference body.
The following examples illustrate it.

\begin{example}\label{nochar}Let $L\in\K^n_s$ have dimension at most $n-1$. Then, the operator
$$\func{\di}{\K^n}{\K^n}{K}{L}$$
is a continuous, Minkowski valuation which is also an $o$-symmetrization and translation invariant. It satisfies RS but it is not $\GL(n)$-covariant. Further, it does not satisfy BM.
\end{example}

\begin{example}
Let $a(K)$ denote either the Steiner point of $K$ (see e.g. \cite[p. 50]{schneider.book14}) or the center of gravity (centroid) of $K$ (see e.g. \cite[p. 314]{schneider.book14}). The operator 
$$
K\mapsto \conv\left((K-a(K)) \,\cup\, (-K+a(-K))\right)
$$ 
satisfies BM and RS (\cite{RS paper 58}). Moreover, it is a continuous $o$-symmetrization. If $a(K)$ is the Steiner point, then the operator is also translation invariant.
\end{example}

\begin{example}\label{without GL}
Let $p\in\R^n$. The operator
$$\func{\di_p}{\K^n}{\K^n}{K}{K-p.}$$
is a continuous Minkowski valuation, which clearly satisfies RS and BM. 
However, $\di_p$ is neither an $o$-symmetrization, nor $\GL(n)$-covariant or translation invariant. 

If $p=\st(K)$, the Steiner point of $K$, then it is further $\O(n)$-covariant but not $\SL(n)$-covariant.
\end{example}

\begin{example}
The operator $$\func{\di}{\K^n}{\K^n}{K}{\vol(K)DK}$$
is a continuous, translation invariant and $\SL(n)$-covariant $o$-symmetrization. This example yields that the $\GL(n)$-covariance imposed in Theorems \ref{th: ti} and \ref{t: GHW class DK} cannot be weakened even to $\SL(n)$.
\end{example}

The previous example shows also that the three conditions continuity, Minkowski valuation and RS together, neither characterize the difference body nor imply $\GL(n)$-covariance.

\begin{example}
Let $B$ be a symmetric convex body with non-empty interior. We define
\[
\func{\di}{\K^n}{\K^n}{K}{DK\cap B.}
\]
The operator $\di$ is an $o$-symmetrization, continuous and translation invariant. It also satisfies the Rogers-Shephard inequality. 

However, $\di$ is neither a Minkowski valuation, nor $\GL(n)$-covariant, nor satisfies BM.
\end{example}

The latter shows that the four conditions: $o$-symmetrization, continuity, translation invariance and Rogers-Shephard inequality neither characterize the difference body nor imply $\GL(n)$-covariance.

Let us also notice that RS is not directly implied by all of the already treated properties, except for the already proven results dealing with the difference body and its relatives, as the following example shows.

\begin{example} 
For $C\in\K^2_s$, the complex difference bodies (see \cite{Abardia DcK}) defined by 
$$h(\D_CK,u)=\int_{S^1}h(\alpha K,u)dS(C,\alpha),\quad u\in S^{n-1},\, K\in\K(\C^n)$$ provide examples of continuous, $o$-symmetrizations, translation invariant Minkowski valuations which satisfy the Brunn-Minkowski inequality but not the Rogers-Shephard inequality. They are neither $\GL(n)$-covariant. 
\end{example}

The following example deals with the continuity condition.

\begin{example}
Let $L\in\K^n_s$ have dimension at most $n-1$. Then, the operator
$$
\di K = \left\{\begin{array}{ll} 
 DK, & \text{  if  } \dim K =n\\
 L, & \text{  otherwise  }  
\end{array}\right.
$$
is an $o$-symmetrization, translation invariant and satisfies both RS and BM. It is however, not continuous.

If $L$ is chosen to be the origin, then it is also $\GL(n)$-covariant, monotonic and 1-homo\-geneous.
\end{example}

\begin{example} 
Let $\omega(K)$ denote the mean width of $K$ (see e.g. \cite[(1.30)]{schneider.book14}). 
Let us consider the operator
$$\func{\di}{\K^n}{\K^n}{K}{B_{\omega(K)},}$$
where $B_{\omega(K)}$ denotes the ball centered at the origin and of radius $\omega(K)$.

The operator $\di$ is a continuous, $o$-symmetrization, translation invariant and Minkowski valuation which satisfies BM. It is also monotonic and homogeneous of degree 1. However, it neither satisfies RS nor is $\GL(n)$-covariant.
\end{example}
Note that if we change, in the last example, the mean width by any other intrinsic volume, then we would lose the Brunn-Minkowski inequality, since $\di$ inherits the homogeneity of the intrinsic volume.

\begin{remark}
We notice that homogeneity assumed together with RS and/or BM implies more precise values for the homogeneity degree. Indeed, let $q$ be the homogeneity degree of $\di$. Then using RS and/or BM we can write
$$c\lambda^n\vol(K)\leq \vol(\di (\lambda K))=\vol(\lambda^q \di K)=\lambda^{qn}\vol(K)\leq c\vol(\lambda K)=C\lambda^n\vol(K).$$
Thus, $q\leq 1$ if RS is assumed, $q\geq 1$ if BM is assumed and, in consequence, $q=1$ if BM and RS are assumed.
Hence, under RS and BM, we can replace $\GL(n)$-covariance assumption by $\SL(n)$-covariance and homogeneity. 
\end{remark}

Next examples prove that RS does not imply, in general, further good properties such us homogeneity or translation invariance.

\begin{example}
The operator
$$\di K=K\cap B_n$$
satisfies RS. It is also continuous and $\O(n)$-covariant. It has neither BM nor is translation invariant.
\end{example}

\begin{example}
Let, for every $K\in\K^n$, $\di K= L$ with $\dim L\leq n-1$. Then, $\di$ is a Minkowski valuation, homogeneous of degree $0$ which satisfies RS. 
\end{example}

\begin{example}\label{noMVal}
Let $\di K= \vol(K)^{1/n}B_n$. It is a continuous $o$-symmetrization satisfying BM and RS. Further, it is translation invariant and homogeneous of degree one. It is clearly not a Minkowski valuation.
\end{example}

\begin{example}
A remarkable example is given by 
$$
K\mapsto L+\vol(K)S
$$
where $S$ is a centered segment and $L$ is an $o$-symmetric $(n-1)$-dimensional convex body so that $\dim(S+L)=n$. 
The operator $\di$ is a continuous, translation invariant Minkowski valuation, and also an $o$-symmetrization which satisfies a Rogers-Shephard and a Brunn-Minkowski type inequality,
since
$$
\vol(L+\vol(K)S)=\vol(K)V(L[n-1],S).
$$
In a forthcoming work, we shall prove that a continuous, translation invariant Minkowski valuation, which is an $o$-symmetrization and satisfies a Rogers-Shephard and a Brunn-Minkowski type inequality is either of the above type or 1-homogeneous. 
\end{example}

Finally we would like to observe that in \cite{schneider74} an example of a translation invariant, $\GL(n)$-covariant valuation which is not continuous is provided. If we slightly modify this example, we can obtain a translation invariant, $\GL(n)$-covariant valuation which is also $o$-symmetrization, 1-homogeneous and satisfies BM, but is not continuous.
Let $SK$ denote the sum of the segments of the boundary of $K$, all centered at the origin. Then $K \mapsto DK+SK$ has the stated properties.


\begin{thebibliography}{10}

\bibitem{Abardia DcK}
J. Abardia. 
\newblock Difference bodies in complex vector spaces.
\newblock \emph{J. Funct. Anal.}, 263 (11), 3588--3603, 2012.

\bibitem{gardner.book06}
R.~J.~Gardner.
\newblock {\em Geometric tomography}, volume~58 of {\em Encyclopedia of
  Mathematics and its Applications}.
\newblock Cambridge University Press, Cambridge, second edition, 2006.

\bibitem{gardner.hug.weil1}
R. J. Gardner, D. Hug, and W. Weil.
\newblock Operations between sets in geometry.
\newblock  \emph{J. Europ. Math. Soc.}, 15, 2297--2352, 2013.

\bibitem{gardner.hug.weil2}
R. J. Gardner, D. Hug, and W. Weil.
\newblock The Orlicz-Brunn-Minkowski theory: a general framework, additions, and inequalities.
\newblock  \emph{J. Differential Geom.}, 97, 427--476, 2014.

\bibitem{hadwiger}
H. Hadwiger.
\newblock {\em Vorlesungen \"uber {I}nhalt, {O}berfl\"ache und
  {I}soperimetrie}.
\newblock Springer-Verlag, Berlin, 1957.

\bibitem{ludwig}
M. Ludwig.
\newblock Minkowski valuations. 
\newblock  \emph{Trans. Amer. Math. Soc.}, 357 (10), 4191--4213 (electronic), 2005.

\bibitem{milman.rotem}
V. Milman and L. Rotem.
\newblock Characterizing addition of convex sets by polynomiality of volume and by the homothety operation.
\newblock {\em Commun. Contemp. Math.}, 17 (3), 1450022, 22 pp, 2015.

\bibitem{rogers.shephard}
C.~A.~Rogers and G.~C.~Shephard.
\newblock The difference body of a convex body.
\newblock {\em Arch. Math.}, 8, 220--233, 1957.

\bibitem{RS paper 58}
C.~A.~Rogers and G.~C.~Shephard.
\newblock Convex bodies associated with a given convex body.
\newblock {\em J. London Math. Soc.}, 33, 270--281, 1958.

\bibitem{schneider74}
R. Schneider.
\newblock Equivariant endomorphisms of the space of convex bodies.
\newblock \emph{Trans. Amer. Math. Soc.}, 194, 53--78, 1974.

\bibitem{schneider.book14}
R. Schneider.
\newblock \emph{Convex bodies: the {B}runn-{M}inkowski theory}, second expanded edition,
\newblock Encyclopedia of Mathematics and its Applications, vol. 151, Cambridge University Press, Cambridge, 2014. 

\bibitem{wannerer.equiv}
T. Wannerer.
\newblock $\GL(n)$ equivariant Minkowski valuations.
\newblock \emph{Indiana Univ. Math. J.}, 60 (5), 1655--1672, 2011.

\end{thebibliography}
\end{document}